\documentclass[11pt, a4paper]{amsart}

\RequirePackage{amssymb,amsfonts, amsmath, amsthm,latexsym}
\RequirePackage[british]{babel}
\RequirePackage{version,graphicx,times,titletoc,dirtytalk}
\RequirePackage[utf8]{inputenc}
\RequirePackage{hyperref}

\newtheorem {thm}{Theorem}

\newtheorem {lem}[thm]{Lemma}

\theoremstyle{definition}

\newtheorem {rem}[thm]{Remark}
\newtheorem {exa}[thm]{Example}

\DeclareMathOperator{\Gal}{Gal}

\newcommand{\Q}{\mathbb{Q}}

\newcommand{\ch}{\mathrm{char}}

\renewcommand{\geq}{\geqslant}
\renewcommand{\leq}{\leqslant}

\newcommand{\group}[1]{\left\langle#1\right\rangle}

\newcommand{\zl}[2][\ell]{\zeta_{{#1}^{#2}}}

\title{The entanglement of radicals}
\author{Chi Wa Chan, Antigona Pajaziti, Flavio Perissinotto and Antonella Perucca}
\date{}

\keywords{Kneser theory, Kummer theory, radicals, entanglement.}
\subjclass[2000]{Primary:  11R20; Secondary: 11R18.}

\begin{document}

\begin{abstract}
In this work we achieve a full understanding of the so-called \emph{entanglement} of radicals, showing that --- over any field --- there are extremely few additive relations among radicals. Our results complete a famous theorem by Kneser from 1975 on the linear independence of radicals and solve a problem discussed by Lenstra in 2006.
\end{abstract}

\maketitle

\section{Introduction}

Let $K$ be a field (for which we fix an algebraic closure $\overline{K}$) and consider a multiplicative group $G$ of radicals of $K$, that is a group generated by $K^\times$ and by elements in $\overline{K}$ that have some power in $K^\times$. Clearly, the multiplicative relations among the radicals in $G$ are encoded in the group structure. We are interested in the additive relations among the radicals in $G$ (also called \emph{entanglement}) that become relevant  when we consider the field $K(G)$. To study the additive relations among radicals we may suppose without loss of generality  that the index $|G:K^\times|$ is finite. Then we can ``measure'' the additive relations by comparing this index and the degree of the extension  $K(G)/K$. Indeed, for radicals that are dependent (in the sense that they give rise to additive relations) the  degree $[K(G):K]$ is smaller than the index $|G:K^\times|$.

Roots of unity are radicals, and $K$-linear relations among them constitute one first type of entanglement, which we call \emph{cyclotomic entanglement}. The basic relations are the following: if $\zeta_n$ is a root of unity of order $n$, then we have
$$1+\zeta_n+\zeta_n^2+\cdots +\zeta_n^{n-1}=0\,.$$
Over $\mathbb Q$ the above relations (and those generated by them) are all the additive relations among roots of unity, but there are more relations involving $\zeta_n$ for a field $K$ such that the degree of the cyclotomic extension $K(\zeta_n)/K$ is less than $\varphi(n)$.
For example, if $K=\mathbb Q(\sqrt{5})\subset \mathbb C$ and $\zeta_5=e^{2\pi i/5}$, then we have the $K$-linear relation
$$\sqrt{5}\cdot 1 -\zeta_5 + \zeta_5^2 + \zeta_5^3 - \zeta_5^4=0\,.$$
We remark that for a number field $K$ there are only finitely many $K$-linear relations among roots of unity which generate all additive relations: this is because there is a constant $c_K$ such that the intersection of $K$ with $\mathbb Q(\zeta_\infty)$ (the largest cyclotomic extension of $\mathbb Q$) is contained in $\mathbb Q(\zeta_{c_K})$. Thus the cyclotomic entanglement is rather harmless and measuring it  amounts to determining the degree of $K\cap \mathbb Q(\zeta_\infty)$.

There can be further \emph{entangled radicals}, for example the above relation 
$$\sqrt{5}=\zeta_5 - \zeta_5^2 - \zeta_5^3 + \zeta_5^4$$
is a $\mathbb Q$-linear relation for the radical $\sqrt{5}$ involving roots of unity.
More generally, all square-roots of rational numbers are contained in a cyclotomic extension of $\mathbb Q$. To generate the corresponding  entanglement we may take 
$$\sqrt{2}=\zeta_8+\zeta_8^{7}$$
and for any odd prime number $p$ express $\sqrt{p}$ as a $\mathbb Q$-linear combination of $4p$-th roots of unity with a Gauss sum that -- with appropriate root choices -- can be written as follows:
$$\sqrt{p}=(-1)^{(p-1)/4}\cdot  \sum_{i=1}^{p-1} \left( \frac{i}{p} \right) \zeta_p^i\,.$$
Over a field $K$ different than $\mathbb Q$ there could be more entanglement of this type, which we call \emph{Kummer entanglement}, because there can be further radicals that are contained in a cyclotomic extension of $K$. For example, over $K=\mathbb Q(\sqrt{5})$ the square root of $-\frac{5+\sqrt{5}}{8}$ is contained in $\mathbb Q(\zeta_5)$. For the Kummer entanglement, the entangled radicals generate abelian radical extensions of $K$ and we can invoke Schinzel's Theorem on abelian radical extensions (Theorem \ref{Schinzel-abelian}): this  entanglement is due to Kummer extensions of $K$ that are contained in cyclotomic extensions and hence it is well-understood.

Over $\mathbb Q$, an entanglement that is neither cyclotomic nor Kummer is given by the following $\mathbb Q$-linear relation (with the appropriate root choices):
$$\sqrt[4]{-4}=1+\zeta_4\,.$$
This entanglement relation is due to the decomposition in \eqref{4}, which in turn stems from the non-cyclicity of the extension $\mathbb Q(\zeta_8)/\mathbb Q$.

In fact, the relations that we presented completely describe the entanglement  over $\mathbb Q$ (this is also a special case of our results below). In this work we solve the problem of describing the entanglement  over any field. As explained by Lenstra in \cite{Lenstra}, beyond the theoretical interest, the understanding of entanglement is crucial for a designer of a computer algebra system who wishes to do computations with radicals e.g.\ over number fields or function fields (namely, it is the main mathematical obstacle that such a designer needs to  overcome). 

In this work, we achieve a full understanding of the entangled radicals by building on a classical result by Kneser (Theorem \ref{thm-Kneser}). Rather surprisingly, the entanglement is as limited as possible. In a nutshell, for a general field there are no surprises with respect to $\mathbb Q$, and there may just be one element (of the form $1+\zeta_{2^w}$) that plays the role that $\sqrt[4]{-4}$ plays for $\mathbb Q$. Our very general results are presented in the next section.

\section{The main results}

We denote as customary the roots of unity, and we let  $\ch(K)$ be the characteristic of $K$.  
We fix a prime number $\ell\neq \ch(K)$ and denote by $\sqrt[\ell^n]{K^\times}$ the subgroup of $\overline{K}^\times$ consisting of the elements whose $\ell^n$-th power is in $K^\times$ (and define $\sqrt[\ell^\infty]{K^\times}$ as the union of $\sqrt[\ell^n]{K^\times}$ for $n\geq 0$). Similarly, we write $\zeta_{\ell^\infty}$ to mean all roots of unity whose order is a power of $\ell$.
Moreover, we call  $K(\zeta_{2\mathcal P})$ the extension of $K$ that is generated by the roots of unity whose order is $4$ or an odd prime number and it is not divisible by $\ch(K)$. 

Our results show the remarkable fact that the additive relations of radicals are extremely few. This is because by a famous result by Kneser (Theorem \ref{thm-Kneser}) there are no more additive relations over the field $K(\zeta_{2\mathcal P})$. 

\begin{thm}\label{thm-onlyKummer-odd}
Suppose that $\ell$ is odd or that $\zeta_4\in K$. 
Let $t\geq 1$ be the largest integer such that $\zeta_{\ell^{t}}\in K(\zeta_{2\mathcal P})$, or set 
$t=\infty$ if no such largest integer exists. 
If $\zeta_\ell\notin K$, then we have 
\begin{equation}\label{prima0}
K(\zeta_{2\mathcal P})\cap \sqrt[\ell^\infty]{K^\times}=\langle \zeta_{\ell^t}, K^\times \rangle \qquad \text{and} \quad 
\qquad K(\zeta_{2\mathcal P})\cap K(\!\sqrt[\ell^\infty]{K^\times})=K(\zeta_{\ell^t})\,.
\end{equation}

If there is some largest integer $w>0$ such that $\zeta_{\ell^{w}}\in K$, we have
\begin{equation}\label{almostobvious}
K(\zeta_{2\mathcal P})\cap K(\sqrt[\ell^\infty]{K^\times})=  K(\zeta_{\ell^{t}}, K(\zeta_{2\mathcal P})\cap\sqrt[\ell^w]{K^\times})
\end{equation}
and
\begin{equation}\label{obvious}K(\zeta_{2\mathcal P})\cap \sqrt[\ell^\infty]{K^\times}=\langle \zeta_{\ell^{t}}, K(\zeta_{2\mathcal P})\cap \sqrt[\ell^w]{K^\times}\rangle\,.
\end{equation}
If $t\leq 2w$, then the field 
$K(\zeta_{2\mathcal P})\cap K(\!\sqrt[\ell^\infty]{K^\times})=K(\zeta_{2\mathcal P})\cap K(\!\sqrt[\ell^w]{K^\times})$
is the largest subextension of $K(\zeta_{2\mathcal P})/K$ that is Kummer and with exponent a power of $\ell$.
\end{thm}

Notice that the case $\langle \zeta_{\ell^\infty}\rangle \subseteq K$ is not dealt with because setting $w=\infty$ in the above formulas would result in a trivial assertion.
We denote by $\sqrt{K^\times}$ the subgroup of $\overline{K}^\times$ consisting of the elements whose square is in $K^\times$, noticing that $K(\sqrt{K^\times})/K$ is a Kummer extension. 

\begin{thm}\label{thm-onlyKummer-2}
Suppose that $\ell= 2$ and $\zeta_4\notin K$. If $\langle\zeta_{2^\infty}\rangle\subseteq  K(\zeta_{2\mathcal P})$, then we have 
\begin{equation}\label{primo'}
K(\zeta_{2\mathcal P})\cap \sqrt[2^\infty]{K^\times} = \langle \zeta_{2^{\infty}}, K(\zeta_{2\mathcal P})\cap \sqrt{K^\times}\rangle
\end{equation}
and 
\begin{equation}\label{secondo'}
K(\zeta_{2\mathcal P})\cap K(\sqrt[2^\infty]{K^\times}) = K( \zeta_{2^{\infty}}, K(\zeta_{2\mathcal P})\cap\sqrt{K^\times})\,.
\end{equation}
Else, call $w\geq 2$ the largest integer such that $\zeta_{2^{w}}+\zeta_{2^{w}}^{-1}\in K$ and let $t\geq w$ be the largest integer such that $\zeta_{2^t}\in K(\zeta_{2\mathcal P})$. Then we have 
\begin{equation}\label{strongall}
K(\zeta_{2\mathcal P})\cap K(\sqrt[2^\infty]{K^\times}) = K(\zeta_{2^t}, K(\zeta_{2\mathcal P})\cap\sqrt{K^\times})
\end{equation}
and
\begin{equation}\label{strongall2}
K(\zeta_{2\mathcal P})\cap \sqrt[2^\infty]{K^\times} = \langle \zeta_{2^{t}}, 1+\zeta_{2^w}, K(\zeta_{2\mathcal P})\cap \sqrt{K^\times}\rangle\,.
\end{equation}
If $t>w$,  we may omit $1+\zeta_{2^w}$ from the list of generators.
\end{thm}

The further results of this paper are described in Section \ref{sec-results}. We don't consider general radicals in $\sqrt[\infty]{K}$ but only  radicals in $\sqrt[\ell^\infty]{K}$ where $\ell$ is prime. This is sufficient for understanding the additive relations among radicals because of the following property: for every extension $K'$ of $K$ and for every $\alpha\in \!\!\sqrt[\ell^\infty]{K'}$ that is not a root of unity, by Kneser theory (possibly applied over $K'(\zeta_4)$, see Theorem \ref{thm-Kneser}) the degree of $K'(\alpha)/K'$ is a power of $\ell$, leading to pairwise coprime degrees for different $\ell$'s.

The proofs of our results rely on two famous theorems: Kneser's theorem on the linear independence of radicals and Schinzel's theorem on abelian radical extensions, see Theorems \ref{thm-Kneser} and \ref{Schinzel-abelian} respectively.

\section{Classical theories of radicals}\label{sec:classic}

Kummer theory concerns the field extensions generated by radicals that satisfy the following condition: any finite subextension is Galois and the exponent of its Galois group is the order of a root of unity contained in the base field. We refer to \cite[Ch. VI \S 8]{LangAlgebra} or \cite{Birch} for an  introduction to Kummer theory. 

Let $K$ be a field, and $\overline{K}$ an algebraic closure of $K$. Let $\Gamma$ be a subgroup of $\bar{K}^\times$ containing $K^\times$ such that $\Gamma/K^\times$ is finite and with order coprime to $\ch(K)$ (in particular, the extension $K(\Gamma)/K$ is separable). We then have the following: 

\begin{thm}[Kummer theory]\label{Kummer-theorem}
 Suppose that the exponent of the group $\Gamma/K^\times$ is the order of a root of unity in $K$. Then $K(\Gamma)/K$ is a Galois extension and we have $$\Gal(K(\Gamma)/K)\simeq \Gamma/{K^\times}
\,.$$
\end{thm}

On the other hand, Kneser theory is the theory about the linear independence of radicals that is based on the following result, see \cite[Satz]{Kneser}.

\begin{thm}[Kneser's theorem on the linear independence of radicals]\label{thm-Kneser}
Suppose that $\zeta_q\in K$ holds for every odd prime $q\neq \ch(K)$ such that $\zeta_q\in \Gamma$. Moreover, if $\ch(K)\neq 2$, suppose that $\zeta_4\in K$ if $1+\zeta_4$ or $1-\zeta_4$ is in $\Gamma$. Then we have 
$$[K(\Gamma):K]=\lvert \Gamma:K^\times\rvert \,.$$ 
\end{thm}

The condition in Kneser's theorem relates to \cite[Theorem 9.1, Ch.VI]{LangAlgebra}:

\begin{thm}\label{thm-Lang}
Let $a\in K^\times$ and $n>1$. The polynomial $x^n-a$ is irreducible in $K[x]$ if for all prime numbers $q\mid n$ we have $a\notin K^{\times q}$ and, in case $\ch(K)\neq 2$ and $\zeta_4\notin K$ and $4\mid n$, we additionally have $a\notin -4 K^{\times 4}$.    
\end{thm}

The reason for the additional assumption for odd characteristic  is the decomposition
\begin{equation}\label{4}
(x^4+4)=(x^2+2x+2)(x^2-2x+2)\,.
\end{equation}
Indeed, the roots of this polynomial are the fourth roots of $-4$: the squareroots of $-4$ generate the field $\mathbb Q(\zeta_4)$ and the fourth roots of $-4$ also generate that field because they are $\pm(1+\zeta_4)$ and $\pm(1-\zeta_4)$. 

We also rely on the following result, see \cite[Theorem 2]{Schinzel-ab}:

\begin{thm}[Schinzel's theorem on abelian radical extensions]\label{Schinzel-abelian}
Let $n\geq 1$ be not divisible by $\ch(K)$. 
If $a\in K^\times$, the extension $K(\zeta_{n}, \sqrt[n]{a})/K$ is abelian if and only if $a^{m}=b^{n}$ holds for some $b\in K^{\times}$ and for some $m\mid n$ such that $\zeta_{m}\in K$.
\end{thm}

Finally, we apply \cite[Satz B]{Koch1} by Halter-Koch, that states the following:
\begin{thm}[Halter-Koch's Theorem B]\label{thm-KochB}
Suppose that $[K(\Gamma):K]=\lvert \Gamma: K^\times \rvert$.
If $\ch(K)\neq 2$ and $\zeta_4\notin K$ and $4$ divides the order of $\Gamma/K^\times$ suppose moreover that the following condition holds: if $y\in K(\Gamma)$ and $(1+\zeta_4)y\in \Gamma$, then $\zeta_4\in K(y)$ or $\zeta_4\in K(y\zeta_4)$.
Then every field $F$ such that $K\subseteq F \subseteq K(\Gamma)$ is conjugated over $K$ to a field of the form $K(\Gamma_F)$ where $\Gamma_F$ is a subgroup of $\Gamma$ that contains $K^\times$.
\end{thm}

We also mention Schinzel's theorem on the linear independence of radicals \cite{Schinzel-lindep}, that is concerned with the case of maximal degree $[K(\Gamma):K]=\lvert\Gamma:K^\times\rvert$. Moreover, Halter-Koch proves further results with a focus on the case in which the degree is maximal (see e.g.\ \cite[Satz 5] {Koch2}). There is also a vast literature on radical extensions, see for example \cite{BarreraVelez} by Barrera Mora and Vélez, and the book \cite{Albu} by Albu. The theory of \emph{entanglement} was established by Lenstra \cite{Lenstra} and it was later developed by Palenstijn \cite{Palenstijn}, see also \cite{MAMA} (for number fields) by Perucca, Sgobba and Tronto.

\section{Overview of the results}\label{sec-results}

\subsection{Notation}

Let $\ell$ be a prime number different from $\ch(K)$. We suppose that we have $$\Gamma=\langle  K^\times, W_{\ell}, \Gamma_\ell\rangle$$ where  
$W_{\ell}$ is a finite group generated by roots of unity of odd prime order different from $\ell$ and where
$\Gamma_\ell$ is a subgroup of $\Gamma$ containing $K^\times$ such that $\Gamma_\ell/K^\times$ has order a power of $\ell$ (this may include roots of unity of order a power of $\ell$).
For the $\ell$-part of the index we have
$$\lvert\Gamma :K^\times\rvert_{\ell}=\lvert\Gamma_\ell :K^\times\rvert\,.$$

\subsection{Results for \texorpdfstring{$\ell$}{ell} odd}

Suppose that $\ell$ is odd and different from the characteristic of $K$.

\begin{thm}\label{thm-oddinK}
Suppose that $\zeta_\ell\in K$, and set $\Gamma_{\ell, E}:=\Gamma_\ell\cap K(W_\ell)$. Then we have $$K(\Gamma_\ell)\cap K(W_\ell)=K(\Gamma_{\ell, E})$$
and
$$[K(\Gamma):K]= \frac{\lvert\Gamma_\ell:K^\times\rvert \cdot [K(W_\ell):K]}{\lvert\Gamma_{\ell,E}:K^\times\rvert}\,.$$
Supposing additionally that a root of unity $\zeta$ of order a power of $\ell$ is in $K(W_\ell)$ only if it is in $K$, then $K(\Gamma_{\ell,E})/K$ is a Kummer extension.
\end{thm}    

\begin{rem}\label{thm-oddinKz4}
    Theorem \ref{thm-oddinK} still holds if we replace $W_\ell$ with $W'_\ell:=\group{W_\ell,\zeta_4}$,  the proof is completely analogous.
\end{rem}

\begin{thm}\label{thm-notin}
Suppose that $\zeta_\ell\notin K$.
\begin{itemize} 
\item If $\zeta_\ell\notin \Gamma$, then we have
$$[K(\Gamma):K]=\lvert\Gamma_\ell :K^\times\rvert\cdot [K(W_\ell):K]\,.$$
\item If $\zeta_\ell\in \Gamma$, then we have $$[K(\Gamma):K]=\frac{\lvert\Gamma_\ell :K^\times\rvert\cdot [K(W_\ell, \zeta_\ell):K]}{\ell^{\varepsilon}\cdot 
[K(\zeta_{\ell^\tau}):K(\zeta_{\ell})]}\,,$$
where $\tau\geq 1$ is the largest integer such that $\zeta_{\ell^\tau}\in \Gamma \cap K(W_\ell, \zeta_\ell)$, and $\varepsilon$ is the largest integer such that $\zeta_{\ell^\varepsilon}\in \Gamma\cap K(\zeta_\ell)$.
\end{itemize}
\end{thm}

\subsection{Results for \texorpdfstring{$\ell$}{l}\texorpdfstring{$=2$}{=2}} 

Suppose that the characteristic of $K$ is  different from $2$. We call \emph{special case of Kneser's theorem} the following case: $\zeta_4\notin K$, and $1+\zeta_4\in \Gamma$ or $1-\zeta_4\in\Gamma$.

\begin{thm}\label{2cheat}
Exclude the special case of Kneser's theorem, and set $\Gamma_{2, E}:=\Gamma_2\cap K(W_2)$.
Then we have $$K(\Gamma_2)\cap K(W_2)=K(\Gamma_{2, E})$$
and 
$$[K(\Gamma):K]= \frac{\lvert\Gamma_2:K^\times\rvert\cdot [K(W_2):K]}{\lvert\Gamma_{2,E}:K^\times\rvert}\,.$$
Supposing additionally that a root of unity $\zeta$ of order a power of $2$ is in $K(W_2)$ only if it is in $K$, then $K(\Gamma_{2,E})/K$ is a Kummer extension.
\end{thm}    

\begin{thm}\label{2cheat2}
Consider the special case of Kneser's theorem. Then we have
$$[K(\Gamma_2):K]=2 \lvert\langle \Gamma_2, K(\zeta_4)^\times\rangle: K(\zeta_4)^\times\rvert\,.$$
Moreover, setting $\Gamma_{2, E}:=\langle\Gamma_2, K(\zeta_4)^\times\rangle\cap K(W_2, \zeta_4)$, we have 
$$K(\Gamma_2)\cap K(W_2,\zeta_4) =K(\Gamma_{2, E})$$
and $$[K(\Gamma):K]= \frac{[K(\Gamma_2):K]\cdot [K(W_2, \zeta_4):K]}{2\cdot \lvert\Gamma_{2,E}:K(\zeta_4)^\times\rvert}\,.$$
\end{thm}

Also consider the following:

\begin{rem}
Let $K$ be a field of characteristic zero, $\ell$ a prime number,  and $W$ a finite group of roots of unity of order coprime to $\ell$. Suppose that $\ell$ does not ramify in any finite subextension of $K/\mathbb Q$. Then a root of unity $\zeta$ of order a power of $\ell$ is in $K(W)$ only if it is already in $K$. If $K$ is a number field, this is because $K(\zeta)/K$ is totally ramified at $\ell$ while $K(W)/K$ is unramified at $\ell$. In general, we may reduce to number fields: firstly we may restrict to consider the subfield of $K$ consisting of algebraic elements, secondly we may  reduce to a finitely generated field. 
\end{rem}

\section{Proof of the results for the case \texorpdfstring{$\ell$}{l} odd}

Let $\ell$ be an odd prime number different from $\ch(K)$.

\begin{proof}[Proof of Theorem \ref{thm-oddinK}]
Let $w$ be the largest integer such that $\zeta_{\ell^w}\in K$, or set $w=\infty$ if $\zeta_{\ell^n}\in K$ holds for every $n\geq 1$. To prove that $K(\Gamma_{\ell,E})/K$ is a Kummer extension, it suffices to show that $K(\alpha)/K$ is a Kummer extension for every $\alpha\in \Gamma_{\ell,E}$. Fix $\alpha\in \Gamma_{\ell,E}$, and let $n$ be the smallest non-negative integer such that $\alpha^{\ell^n}\in K$. We have to prove that $n\leq w$, so suppose instead that $n>w$.
By Theorem \ref{Schinzel-abelian}, as $\alpha$ is contained in an abelian extension of $K$, we have $\alpha^{\ell^w}\in \langle K^\times, \zeta_{\ell^m}\rangle$ for some minimal non-negative integer $m\leq n$, and we must have $m>w$ because $\alpha^{\ell^w}\notin K^\times$. We deduce that 
$\zeta_{\ell^m}\in \langle K^\times, \alpha^{\ell^w}\rangle$ and hence $\zeta_{\ell^m}\in K(W_\ell)$. The additional assumption implies $\zeta_{\ell^m}\in K$ and hence $m\leq w$, contradiction. 

Now consider the general case. Since $\zeta_\ell\in K$, by Kneser's theorem we have 
\begin{equation}\label{eq-oddin}
[K(\Gamma_{\ell}):K]=\lvert\Gamma_\ell:K^\times\rvert\quad \text{and}\quad [K(\Gamma_{\ell, E}):K]=\lvert\Gamma_{\ell, E}:K^\times\rvert\,.
\end{equation}
So we are left to prove 
\begin{equation}\label{eq-oddingo}
K(\Gamma_\ell)\cap K(W_\ell)=K(\Gamma_{\ell, E})\,,
\end{equation}
the inclusion $\supseteq$ being clear. Letting $F=K(\Gamma_{\ell, E})$, it suffices to prove that $F(\Gamma_\ell)\cap F(W_\ell)=F$.
By \eqref{eq-oddin} the degree of $F(\Gamma_\ell)/F$ is a power of $\ell$. Since $F(W_\ell)/F$ is abelian, it suffices to prove that $F(\Gamma_\ell)/F$ has no subextension $L/F$ of degree $\ell$ contained in $F(W_\ell)$. As $\zeta_\ell\in F$, such an extension would be a Kummer subextension of $F(\Gamma_\ell)/F$ and hence we would have $L=F(\gamma)$ for some $\gamma\in \Gamma_\ell$. Since $\gamma\in L\subseteq K(W_\ell)$, this would contradict $F(\Gamma_{\ell, E})=F$.
\end{proof}

\begin{lem}\label{helplem}
Suppose that $\zeta_\ell\notin K$ and $\zeta_\ell\in \Gamma$. Letting $\tau\geq 1$ be the largest integer such that $\zeta_{\ell^{\tau}}\in \Gamma\cap K(W_\ell, \zeta_\ell)$, we have
\begin{equation}\label{first}
K(\Gamma_\ell)\cap K(W_\ell,\zeta_\ell)= K(\zeta_{\ell^\tau})\,.
\end{equation}
Moreover, let $\varepsilon\geq 1$ be the largest integer such that $\zeta_{\ell^\varepsilon}\in \Gamma\cap K(\zeta_\ell)$.
We have 
$$[K(\Gamma_\ell) :K(\zeta_\ell) ] = \lvert\Gamma_\ell  :K^\times\rvert\cdot \ell^{-\varepsilon}$$  
and
\begin{equation}\label{lemmetto}
\lvert\Gamma_\ell \cap K(\zeta_\ell)^\times :K^\times\rvert=\ell^{\varepsilon}\,.
\end{equation}
\end{lem}
\begin{proof} 
We first prove \eqref{first}. The inclusion $\supseteq$ holds because $\zeta_{\ell^\tau}\in \Gamma_{\ell}$, so it suffices to consider $F:=K(\zeta_{\ell^\tau})$ and prove $$ F(\Gamma_\ell)\cap F(W_\ell)\subseteq F\,.$$ Over $F$ we can apply Kneser's theorem to $\langle \Gamma_\ell, F^\times\rangle$. If $F(\Gamma_\ell)\cap F(W_\ell)$ is larger than $F$, it contains a subfield $L$ such that $L/F$ has degree $\ell$ (hence it is a Kummer extension). So we have $L=F({\gamma})$ for some $\gamma\in \Gamma_\ell$. Notice that $\gamma\in F(W_\ell)=K(W_\ell, \zeta_\ell)$. Let $m\geq 1$ be minimal such that $\gamma^{\ell^m}\in K^\times$. Since $K(\zeta_{\ell^m}, \gamma)$ is abelian, by Theorem \ref{Schinzel-abelian} we deduce that $\gamma\in \langle \zeta_{\ell^n}, K^\times \rangle$ holds for some minimal positive integer $n\leq m$. So we have $$\zeta_{\ell^n}\in \langle \gamma, K^\times \rangle \subseteq \Gamma \cap K(W_\ell, \zeta_\ell)$$ and hence $n\leq \tau$. We deduce that $\gamma\in F$ and $L=F$, contradiction.

By Kneser's theorem over $K(\zeta_\ell)$ we have 
$$[K(\Gamma_\ell) :K(\zeta_\ell)]=\lvert\langle \Gamma_\ell, K(\zeta_\ell)^\times \rangle :K(\zeta_\ell)^\times\rvert=\lvert \Gamma_\ell : \Gamma_\ell\cap K(\zeta_\ell)^\times\rvert\,.$$
So to conclude it suffices to prove \eqref{lemmetto}. Notice that
$\ell^\varepsilon$ divides the index in \eqref{lemmetto} because $\zeta_{\ell^\varepsilon}\in \Gamma_\ell \cap K(\zeta_\ell)^\times$ and $\zeta_\ell\notin K^\times$.
It then suffices to prove that for every $\alpha\in \Gamma_\ell \cap K(\zeta_\ell)^\times$ we have $\alpha\in \langle \zeta_{\ell^n}, K^\times\rangle$ for some integer $n\geq 0$ (taking $n$ minimal, we have $n\leq \varepsilon$ because $\zeta_{\ell^n}\in \langle \alpha, K^\times\rangle$). This is a consequence of Theorem \ref{Schinzel-abelian} because we have $\alpha^{\ell^m}\in K^\times$ for some $m\geq 0$ and $\alpha$ is contained in an abelian extension of $K$ (hence $\alpha^{\ell^m}\in K^{\times \ell^m}$).
\end{proof}

\begin{proof}[Proof of Theorem \ref{thm-notin}]
Suppose first that $\zeta_\ell\notin \Gamma$. By Kneser's theorem we have 
$[K(\Gamma_\ell) :K]=\lvert\Gamma_\ell  :K^\times\rvert$ hence it suffices to prove that $K(W_\ell)\cap K(\Gamma_\ell)=K$. The extension $K(W_\ell)/K$ is abelian while $K(\Gamma_\ell)/K$ has degree a power of $\ell$ by Kneser's theorem applied to $\Gamma_\ell$. So it suffices to prove that $K(\Gamma_\ell)$ has no subextension $L/K$ of degree $\ell$ that is abelian. By Theorem \ref{thm-KochB} applied to $\Gamma_\ell$ the field $L$ is conjugated and thus equal to $K(\gamma)$ for some $\gamma\in \Gamma_\ell\setminus K^\times$. By Kneser's theorem applied to $\langle \gamma, K^\times \rangle$ we deduce that $\gamma^\ell\in K^\times$. Since the extension $K(\zeta_{\ell}, \gamma)/K$ is  abelian, Theorem \ref{Schinzel-abelian} implies $\gamma^{\ell}\in K^{\times \ell}$. So we have $\gamma\in \Gamma_\ell \cap \langle \zeta_{\ell}, K^{\times}\rangle$, contradicting that $\gamma\notin K^\times$ and $\zeta_\ell\notin \Gamma_\ell$.

Now consider the case $\zeta_\ell \in \Gamma$ (equivalently, $\zeta_\ell \in \Gamma_\ell$). We can apply Theorem \ref{thm-oddinK} to $\Gamma_\ell':=\langle \Gamma_\ell, K(\zeta_\ell)^\times \rangle$ over $K(\zeta_\ell)$, setting $\Gamma'_{\ell, E}:=\Gamma'_\ell \cap K(W_\ell, \zeta_{\ell})$.
We get 
\[[K(\Gamma):K(\zeta_\ell)]= \frac{|\Gamma'_\ell:K(\zeta_\ell)^\times|\cdot [K(W_\ell, \zeta_\ell):K(\zeta_\ell)]}{|\Gamma'_{\ell,E}:K(\zeta_\ell)^\times|}\]
and hence 
\[[K(\Gamma):K]=\frac{|\Gamma_\ell:K^\times|\cdot [K(W_\ell, \zeta_\ell):K]}{|\Gamma'_{\ell,E}:K(\zeta_\ell)^\times|\cdot|\Gamma_\ell \cap K(\zeta_\ell)^\times : K^\times|}\,.\]
Recalling \eqref{lemmetto} it suffices 
to show $|\Gamma'_{\ell,E}:K(\zeta_\ell)^\times|  = [K(\zeta_{\ell^\tau}):K(\zeta_\ell)]$. By Kneser's theorem over $K(\zeta_\ell)$ we have 
 \[ |\Gamma'_{\ell,E} :K(\zeta_\ell)^\times|=[K(\Gamma'_{\ell, E}):K(\zeta_\ell)]\]
so we may conclude by proving $K(\zeta_{\ell^\tau})=K(\Gamma'_{\ell, E})$. The inclusion $\subseteq$ is because $\zeta_{\ell^\tau}\in \Gamma_\ell \cap K(W_\ell, \zeta_\ell)$. The other inclusion is because 
$K(\Gamma'_{\ell, E})\subseteq K(\Gamma_\ell)\cap K(W_\ell, \zeta_\ell)$ (as $K(\Gamma'_\ell)=K(\Gamma_\ell)$) and this intersection equals $K(\zeta_{\ell^\tau})$ by \eqref{first}.
\end{proof}

\begin{proof}[Proof of Theorem \ref{thm-onlyKummer-odd} for $\ell$ odd.]
The last assertion concerning the special case $t\leq 2w$ follows from \eqref{obvious} and \eqref{almostobvious1}, considering that $\zeta_{\ell^t}\in \sqrt[\ell^w]{K^\times}$. 

To avoid a case distinction, we set $w=0$ if $\zeta_\ell\notin K$.
We first prove
\begin{equation}\label{obvious1}K(\zeta_{2\mathcal P})\cap \sqrt[\ell^\infty]{K^\times}\subseteq \langle \zeta_{\ell^{t+w}}, \sqrt[\ell^w]{K^\times} \rangle\end{equation}
\begin{equation}\label{almostobvious1}K(\zeta_{2\mathcal P})\cap K(\sqrt[\ell^\infty]{K^\times})=
K(K(\zeta_{2\mathcal P})\cap \sqrt[\ell^\infty]{K^\times})\,.
\end{equation}
Let $\alpha\in K(\zeta_{2\mathcal P})\cap\sqrt[\ell^\infty]{K^\times}$.
Since $\alpha$ is contained in an abelian extension of $K$, by Theorem \ref{Schinzel-abelian} there is some non-negative integer $n$ such that $\alpha^{\ell^w} \zeta_{\ell^n}\in K^\times$. We deduce that $\zeta_{\ell^n}\in K(\zeta_{2\mathcal P})$ and hence $n\leq t$, so \eqref{obvious1} follows.
If $w=0$, \eqref{obvious1} implies 
$K(\zeta_{2\mathcal P})\cap \sqrt[\ell^\infty]{K^\times}= \langle \zeta_{\ell^{t}}, {K^\times} \rangle$. The second equality in \eqref{prima0} will then follow from \eqref{almostobvious1}.

The inclusions $\supseteq$ in \eqref{almostobvious}, \eqref{obvious} and \eqref{almostobvious1} are immediate, and to prove the other inclusions we may replace $\sqrt[\ell^\infty]{K^\times}$ by a subgroup $\Gamma_{\ell}\supseteq K^\times$ such that $\Gamma_{\ell}/K^\times$ is finite. Moreover, we may replace $K(\zeta_{2\mathcal P})$
by a subfield $K(W_\ell, \zeta_4, \zeta_{\ell})\ni \zeta_{\ell^t}$ where $W_\ell$ is a group generated by finitely many roots of unity of odd prime order different from $\ell$.
Set $W'_\ell:=\langle W_\ell, \zeta_4\rangle$. In view of Remark \ref{thm-oddinKz4},  Theorem \ref{thm-oddinK} applied to $\langle\Gamma_\ell, K(\zeta_{\ell^t})^\times\rangle$ over $K(\zeta_{\ell^t})$ gives
$$K(W'_\ell, \zeta_\ell)\cap K(\Gamma_\ell)\subseteq K(\langle\Gamma_\ell, K(\zeta_{\ell^t})^\times\rangle\cap K(W'_\ell, \zeta_{\ell}))\,.$$
Over $K(\zeta_{\ell^t})^\times$, the elements of 
$\langle\Gamma_\ell, K(\zeta_{\ell^t})^\times\rangle\cap K(W'_\ell, \zeta_\ell)$ are generated by elements in 
$\Gamma_\ell\cap K(W'_\ell, \zeta_\ell)\subseteq K(\zeta_{2\mathcal P})\cap \!\sqrt[\ell^\infty]{K^\times}$ and we conclude the proof of \eqref{almostobvious1}  because $\zeta_{\ell^t}\in K(\zeta_{2\mathcal P})\cap \!\sqrt[\ell^\infty]{K^\times}$.

We now prove \eqref{obvious}, where we may suppose that $t$ is finite (the case $t=\infty$ being obvious) and hence $K$ has characteristic zero by Remark \ref{ff}. Notice that 
the containment $\supseteq$ is clear.
From \eqref{obvious1} we deduce that $$K(\zeta_{2\mathcal P})\cap \sqrt[\ell^\infty]{K^\times}\subseteq \langle \zeta_{\ell^{t+w}}, K(\zeta_{\ell^{t+w}}, \zeta_{2\mathcal P})\cap \sqrt[\ell^w]{K^\times}\rangle\,.$$

Let $\alpha\in K(\zeta_{2\mathcal P}) \cap \sqrt[\ell^\infty]{K^\times}$ and write $\alpha=\zeta_{\ell^{t+w}}^m\beta$ where $\beta\in K(\zeta_{\ell^{t+w}}, \zeta_{2\mathcal P})\cap \sqrt[\ell^w]{K^\times}$ and $m\geq 1$.

By Kummer theory (because $K(\zeta_{2\mathcal{P}},\zeta_{\ell^{t+w}})/K$ is abelian and we investigate a Kummer subextension) we may write $\beta=\beta'\gamma$ so that $\beta'\in K(\zeta_{\ell^{t+w}})\cap \sqrt[\ell^w]{K^\times}$ and $\gamma\in K(\zeta_{2\mathcal P})\cap \sqrt[\ell^w]{K^\times}$. 
We may suppose w.l.o.g.\ that $\gamma=1$. 
So we have
$$K(\alpha)\subseteq K(\zeta_{\ell^{t+w}})\cap K(\zeta_{2\mathcal P})=K(\zeta_{\ell^t})\,.$$

We have $\beta\in K(\zeta_{\ell^{2w}})$ because $K(\beta)$ is a subextension of $K(\zeta_{\ell^{t+w}})/K$ with exponent dividing $\ell^w$.
From $K(\alpha)\subseteq K(\zeta_{\ell^{t}})$ we deduce that $t+w-v_\ell(m)\leq \max(t,2w)$. If $t\geq 2w$ we may conclude because $\alpha\in \langle \zeta_{\ell^t}, \beta\rangle\cap K(\zeta_{2\mathcal P})=\langle \zeta_{\ell^t}, \langle\beta\rangle\cap K(\zeta_{2\mathcal P})\rangle $. Else, we conclude because $\alpha\in \langle \zeta_{\ell^{2w}}, \beta\rangle\cap K(\zeta_{2\mathcal P})\subseteq \sqrt[\ell^w]{K}\cap K(\zeta_{2\mathcal P})$.

Notice that \eqref{almostobvious} can be obtained by combining \eqref{almostobvious1} and \eqref{obvious}.
\end{proof}

\section{Proof of the results for the case \texorpdfstring{$\ell$}{l}=2}

Now we consider the results for $\ell=2$. 

\begin{proof}[Proof of Theorem \ref{2cheat}]
This is the analogue of Theorem \ref{thm-oddinK}. Beyond the special case of Kneser's theorem, we may reason as done in the proof of Theorem \ref{thm-oddinK}  for the case $\zeta_\ell\in K$.
\end{proof}

\begin{proof}[Proof of Theorem \ref{thm-onlyKummer-odd} in case $\ell=2$.]
Since $\zeta_4\in K$, we may proceed as in the case $\ell$ odd and $\zeta_\ell\in K$, relying on Theorem \ref{2cheat} in place of Theorem \ref{thm-oddinK}.
\end{proof}

\begin{proof}[Proof of Theorem \ref{2cheat2}]
Since we are in the special case of Kneser's theorem we have in particular $\zeta_4\notin K$ and 
$\zeta_4\in K(\Gamma_2)$. By Theorem \ref{2cheat} applied to $\Gamma_2':=\langle \Gamma_2, K(\zeta_4)^\times\rangle$ over $K(\zeta_4)$ we obtain 
$$K(\Gamma_2)\cap K(W_2, \zeta_4) =K(\Gamma_{2, E})$$
and 
$$[K(\Gamma):K(\zeta_4)]= \frac{\lvert\Gamma'_2:K(\zeta_4)^\times\rvert\cdot [K(W_2, \zeta_4):K(\zeta_4)]}{\lvert\Gamma_{2,E}:K(\zeta_4)^\times\rvert}\,.$$
It then suffices to prove 
$$[K(\Gamma_2):K(\zeta_4)]= \lvert\Gamma'_2: K(\zeta_4)^\times \rvert$$
which follows by Kneser's Theorem applied to $\Gamma_2'$ over $K(\zeta_4)$.
\end{proof}

\begin{rem}\label{ff}
If $p$ is a prime and $\ell\neq p$ is a prime, then $\mathbb F_p(\zeta_{2\mathcal P})$ contains $\mathbb F_p(\zeta_{\ell^\infty})$. Indeed, the field $\mathbb F_p(\zeta_{\ell^\infty})$ is the compositum of $\mathbb F_p(\zeta_{\ell})$ and of 
all extensions of $\mathbb F_p$ whose degree is a power of $\ell$. Moreover, by Zygsmondy's Theorem \cite{Zsi}, for every $m\geq 3$ there is a prime $q\neq p$ such that the multiplicative order of $(p \bmod q)$ equals $m$, which implies $[\mathbb F_p(\zeta_q):\mathbb F_p]=m$.
\end{rem}

\begin{rem}\label{friends}
With the notation of Theorem \ref{thm-onlyKummer-2}, let $w\geq 2$ and suppose that
$\zeta_{2^{w}}+\zeta_{2^{w}}^{-1}\in K$.
Consider the radical 
$$\eta:=\zeta_{2^{w+1}}\sqrt{\zeta_{2^{w}}+\zeta_{2^{w}}^{-1}+2}\in \sqrt[2^\infty]{K^\times}\,.$$
We have
$\eta^2=(1+\zeta_{2^w})^2$ hence $\eta\in \{\pm (1+\zeta_{2^w})\}$ and $K(\eta)= K(\zeta_{4})$.
Notice that 
$$\zeta_{2^{w+1}}^{-1}\sqrt{\zeta_{2^{w}}+\zeta_{2^{w}}^{-1}+2}\in \langle \eta, \zeta_{2^w}, K^\times\rangle$$ and that, in general, the ratio between a radical and its negative is in $K^\times$.
\end{rem}

\begin{exa}
With the notation of Theorem \ref{thm-onlyKummer-2}, if $K=\mathbb Q(\sqrt{6})$, then $t=3$ and $\zeta_8\notin K(\zeta_4)$.    
\end{exa}

\begin{proof}[Proof of Theorem \ref{thm-onlyKummer-2}]
Let $s$ be the largest element in $\mathbb Z \cup \{\infty\}$ such that $\langle \zeta_{2^s}\rangle\subseteq K(\zeta_{2\mathcal P})$ (and let $s+1=\infty$ if $s=\infty$).
We first prove 
\begin{equation}\label{primo}
K(\zeta_{2\mathcal P})\cap \sqrt[2^\infty]{K^\times} \subseteq \langle \zeta_{2^{s+1}}, \sqrt{K^\times}\rangle\,.
\end{equation}
Fix $\alpha \in K(\zeta_{2\mathcal P}) \cap \sqrt[2^\infty]{K^\times}$. To investigate $\alpha$ we may replace $\sqrt[2^\infty]{K^\times}$ by a subgroup $\Gamma_2\supseteq K^\times$ such that $\Gamma_2/K$ is finite and contains $1\pm \zeta_4$, and we may replace $K(\zeta_{2\mathcal P})$ by an extension of the form $K(\zeta_4, W_2)$ where $W_2$ is generated by finitely many roots of unity that have odd prime order.
Then $\alpha\in \Gamma_2\cap K(W_2, \zeta_4)$. Since $\alpha$ is contained in an abelian extension of $K$ by Theorem \ref{Schinzel-abelian} (since $\zeta_4\notin K$) we have $\alpha^2 \cdot \zeta_{2^n}\in K^\times$ for some minimal $n\geq 0$.
We deduce that $\zeta_{2^n}\in \langle \alpha^2, K^\times \rangle \subseteq \Gamma_2\cap K(W_2, \zeta_4)$ hence $n\leq s$. We deduce that 
$\alpha\in \langle \zeta_{2^{s+1}}, \sqrt{K^\times}\rangle$.
We now prove
\begin{equation}\label{secondo}
K(\zeta_{2\mathcal P})\cap K(\sqrt[2^\infty]{K^\times}) \subseteq K( \zeta_{2^{s+1}}, \sqrt{K^\times})\,.
\end{equation}
It suffices to show that, if $W_2$ and $\Gamma_2$ are as above, we have 
$$K(W_2, \zeta_4)\cap K(\Gamma_2) \subseteq K( \zeta_{2^{s+1}}, \sqrt{K^\times})\,.$$
By Theorem \ref{2cheat2} we have $$K(W_2,\zeta_4)\cap K(\Gamma_2)=K(\Gamma_{2, E})\quad \text{where}\quad \Gamma_{2, E}:=\langle\Gamma_2, K(\zeta_4)^\times\rangle\cap K(W_2, \zeta_4)\,.$$
We may conclude because the group $\Gamma_{2, E}$ is generated by $K(\zeta_4)^\times\subseteq K(\sqrt{K^\times})$ and by elements  
in $\Gamma_2\cap K(W_2, \zeta_4)$ which, as shown above, are in $\langle \zeta_{2^{s+1}}, \sqrt{K^\times}\rangle$.

The assertion for $s=\infty$ is a  consequence of \eqref{primo} and \eqref{secondo}. Now suppose that $s$ is finite. By Remark \ref{ff} the field $K$ has characteristic zero. Notice that \eqref{primo} implies 
\begin{equation}\label{last}
K(\zeta_{2\mathcal P})\cap \sqrt[2^\infty]{K^\times}\subseteq \langle \zeta_{2^{s+1}}, \sqrt{K^\times}\cap K(\zeta_{2\mathcal P}, \zeta_{2^{s+1}})\rangle\,.
\end{equation}
Fix an embedding $\Q(\zeta_{2^\infty}) \hookrightarrow \overline{K}$ and write $K_0:=K\cap\Q(\zl[2]{\infty})$. Let $\alpha\in K(\zeta_{2\mathcal P}) \cap \sqrt[2^\infty]{K^\times}$ and write $\alpha=\zeta_{2^{s+1}}^m\beta$ where $\beta\in \sqrt{K^\times}\cap K(\zeta_{2\mathcal P}, \zeta_{2^{s+1}})$ and $m\geq 1$.

By Kummer theory (since $K(\zeta_{2\mathcal{P}},\zl[2]{s+1})/K$ is abelian, a subextension of degree $2$ is contained in the compositum of two subextensions of degree at most $2$ of $K(\zeta_{2\mathcal{P}})/K$ and $K(\zl[2]{s+1})/K$ respectively) we may write $\beta=\beta'\gamma$ so that $\beta'\in \sqrt{K^\times}\cap K(\zeta_{2^{s+1}})$ and $\gamma\in \sqrt{K^\times}\cap K(\zeta_{2\mathcal P})$. 

We now prove \eqref{strongall2}, noticing that the containment $\supseteq$ holds by Remark \ref{friends}.
We may suppose w.l.o.g.\ that $\gamma=1$. 
So we have
$$K(\alpha)\subseteq K(\zeta_{2^{s+1}})\cap K(\zeta_{2\mathcal P})=K(\zeta_{2^s})\,.$$

If $K(\zeta_{2^s})$ strictly contains $K(\zeta_{4})$ or if $K_0$ is not totally real, then the exponent of $K(\zl[2]{s+1})/K$ is divisible by $4$. We deduce that $\beta\in K(\zeta_{2^s})$ because $\beta$ is contained in a subextension of exponent $2$ of $K(\zeta_{2^{s+1}})/K$.
From $K(\alpha)\subseteq K(\zeta_{2^{s}})$ we deduce that $m$ must be even and we may easily conclude.

Now we may suppose that $K(\zeta_{2^s})=K(\zeta_{4})$, that $K_0$ is totally real, and w.l.o.g.\ that $\alpha\notin \sqrt{K^\times}$. So we have $K(\alpha)=K(\zeta_{4})$ and $s=w$. 
By Remark \ref{friends}, the radical $\eta\in \sqrt[2^\infty]{K^\times}$ is such that $K(\eta)=K(\zeta_4)$, and the same holds for $\eta/\zeta_{2^s}$.

If $1\pm \zeta_4\notin \langle \alpha, K^\times \rangle$, then by Kneser's theorem the degree of $K(\alpha)/K$ is $2^{n}$, where $n\geq 2$ is minimal such that $\alpha^{2^n}\in K$. This contradicts $\alpha\in K(\zeta_4)$. 
From this we also deduce that $R\in \langle \eta, K^\times \rangle$ and $R'\in \langle \eta/\zeta_{2^s}, K^\times \rangle$ hold for some $R,R'\in \{1\pm \zeta_4\}$, where $R \neq R'$ because $\eta$ and $\eta/\zeta_{2^s}$ are complex conjugates (for any embedding of the involved radicals inside $\mathbb C$).

Finally suppose that $R\in \langle \alpha, K^\times \rangle$ for some $R\in \{1\pm \zeta_4\}$. If $\alpha\in \langle R, K^\times \rangle$, we may conclude because $\alpha\in \langle \zeta_{2^s}, \eta, K^\times\rangle$. Else, up to replacing $\alpha$ by an odd power of it, or replacing $\alpha$ by its reciprocal, we can write $\alpha^{2^d}= Rk_0$ for some $k_0\in K^{\times}$ and for some $d\geq 1$. Writing $R=\zeta_8^{\pm 1}\sqrt{2}$
we get $\alpha=\zeta_{2^{3+d}}^{x} \sqrt[2^d]{\sqrt{2}k_0}$ for some odd integer $x$. Since $\sqrt[2^d]{\sqrt{2}k_0}$ is contained in an abelian extension of $K$, by Theorem \ref{Schinzel-abelian} we have $2 k_0^2\in K^{\times 2^d}$ and hence $\sqrt{2}\in K$. 
Then we have 
$\alpha=\zeta_{2^{3+d}}^{y} \sqrt{k_1}$ for some $k_1\in K^{\times}$ and for some odd integer $y$. If $3+d\leq s$ we deduce that $\alpha\in \langle \zeta_{2^{s}}, K(\zeta_{2\mathcal P})\cap \sqrt{K^\times}\rangle$ and we conclude. Moreover, we cannot have $3+d\geq  s+2$ because $K(\alpha, \sqrt{k_1})/K$ has exponent $2$ while $K(\zeta_{2^{s+2}})/K$ has exponent at least $4$. Finally suppose that $3+d=s+1$. The conditions $K(\alpha)=K(\zeta_4)$ and $\zeta_{2^{s+1}}\in K(\alpha, \sqrt{k_1})$
imply that $K(\sqrt{k_1})$
is either $K(\zeta_{2^{s+1}}+\zeta_{2^{s+1}}^{-1})$ or
$K(\zeta_4(\zeta_{2^{s+1}}+\zeta_{2^{s+1}}^{-1}))$. 
Remarking that $(\zeta_{2^{s+1}}+\zeta_{2^{s+1}}^{-1})^2=\zeta_{2^{s}}+\zeta_{2^{s}}^{-1}+2$, we conclude because
$\alpha\in \langle \zeta_{2^s}, \eta, K^\times\rangle$.

To show \eqref{strongall}, consider the proof of \eqref{secondo} and observe that by 
\eqref{strongall2} we know that $\Gamma_2\cap K(W_2, \zeta_4)$ is contained in $K(\zeta_{2^s}, K(\zeta_{2\mathcal P})\cap \sqrt{K^\times})$.
\end{proof}

\subsection*{Acknowledgements}
Perucca is the main author of this paper, which originated from a discussion between the last two listed authors and which was then the subject of a Kummer theory course at the University of Luxembourg. We thank Fritz Hörmann for useful remarks. Pajaziti and Perissinotto have been supported by the Luxembourg National Research Fund AFR-PhD 16981197 and PRIDE17/1224660/GPS.

\end{document}